\documentclass{amsart}

\usepackage{amsmath,amsthm,amssymb}
\usepackage{euscript}

\usepackage{mathtools} 

\usepackage{tikz}
\usetikzlibrary{matrix,arrows, calc}
\definecolor{color1}{RGB}{134,14,156}
\definecolor{color2}{RGB}{6,138,39}
\definecolor{color3}{RGB}{12,39,156}
\definecolor{color4}{RGB}{153,0,0}
\definecolor{mygray}{RGB}{192, 192, 192}

\tikzset{
        hatch distance/.store in=\hatchdistance,
        hatch distance=10pt,
        hatch thickness/.store in=\hatchthickness,
        hatch thickness=2pt
    }
    \makeatletter
    \pgfdeclarepatternformonly[\hatchdistance,\hatchthickness]{flexible hatch}
    {\pgfqpoint{0pt}{0pt}}
    {\pgfqpoint{\hatchdistance}{\hatchdistance}}
    {\pgfpoint{\hatchdistance-1pt}{\hatchdistance-1pt}}%
    {
        \pgfsetcolor{\tikz@pattern@color}
        \pgfsetlinewidth{\hatchthickness}
        \pgfpathmoveto{\pgfqpoint{0pt}{0pt}}
        \pgfpathlineto{\pgfqpoint{\hatchdistance}{\hatchdistance}}
        \pgfusepath{stroke}
    }

\newtheorem{thm}[subsection]{Theorem}
\newtheorem{prop}[subsection]{Proposition}
\newtheorem{cor}[subsection]{Corollary}

\theoremstyle{definition}  

\newtheorem{example}[subsection]{Example}
\newtheorem{remark}[subsection]{Remark}

\newcommand{\F}{\mathbb{F}}
\newcommand{\C}{\mathbb{C}}

\newcommand{\Z}{\mathbb{Z}}

\newcommand{\map}{\rightarrow}
\newcommand{\iso}{\cong}
\newcommand{\ol}{\overline}

\newcommand{\Ct}{C \tau} 
\renewcommand{\aa}{\alpha_1} 

\newcommand{\clE}[2][{}]{E_{#2}^{#1}(S^0; BP)}
\newcommand{\motE}[2][{}]{E_{#2}^{#1}(S^{0,0}; BPL)}

\newcommand{\olmotE}[2][{}]{\ol{E}_{#2}^{#1}(S^{0,0}; BPL)}
\newcommand{\quotient}[2]{{\raisebox{.2em}{$#1$}\left/\raisebox{-.2em}{$#2$}\right.}}

\DeclareMathOperator{\colim}{colim}

\usepackage{chemarrow, xspace} 

\renewcommand{\to}{\ensuremath{ \mathrel{ \mkern1.5mu\textrm{\arro\symbol{71}}  \mkern-1.1mu\textrm{\arro\symbol{65}} \mkern+1mu } \xspace }}	
\newcommand{\lto}{\ensuremath{ \mathrel{ \mkern1.5mu\textrm{\arro\symbol{71}} \mkern-2.5mu\textrm{\arro\symbol{71}} \mkern-1.1mu\textrm{\arro\symbol{65}} \mkern+1mu } \xspace }}	

\begin{document}

\title{The structure of motivic homotopy groups}

\author{Bogdan Gheorghe}
\email{gheorghebg@wayne.edu}
\author{Daniel C.\ Isaksen}
\email{isaksen@wayne.edu}
\address{Department of Mathematics\\ Wayne State University\\
Detroit, MI 48202, USA}
\date{\today}

\thanks{The second author was supported by NSF grant DMS-1202213.}



\begin{abstract}
We study the stable motivic homotopy groups $\pi_{s,w}$ 
of the 2-completion of the motivic sphere spectrum over $\C$.
When arranged in the $(s,w)$-plane, these groups
break into four different regions: a vanishing region,
an $\eta$-local region that is entirely known,
a $\tau$-local region that is identical to classical stable homotopy
groups, and a region that is not well-understood.
\end{abstract}

\maketitle

\section{Introduction}
\label{sctn:intro}

This article is concerned with the motivic stable homotopy groups over
$\C$.  More specifically, we consider the motivic stable
homotopy groups $\pi_{s,w}$ of the 2-completion of the motivic sphere
spectrum, where $s$ is the stem and $w$ is the motivic weight.
Motivic completion behaves somewhat differently
than classical completion.  In particular, the homotopy
groups of the completed motivic sphere are not necessarily the same as the
completions of the integral motivic homotopy groups.
In fact, one needs to be careful about completion with respect to the
Hopf map $\eta$ as well 
\cite{HKO11b}.

The bigraded nature of the motivic homotopy 
groups leads to a natural arrangement
in the $(s,w)$-plane.  Our main result is that this plane breaks into
four distinct regions.  These regions are indicated in Figure
\ref{fig:chartintro}.  For clarity, the figure is not to scale.

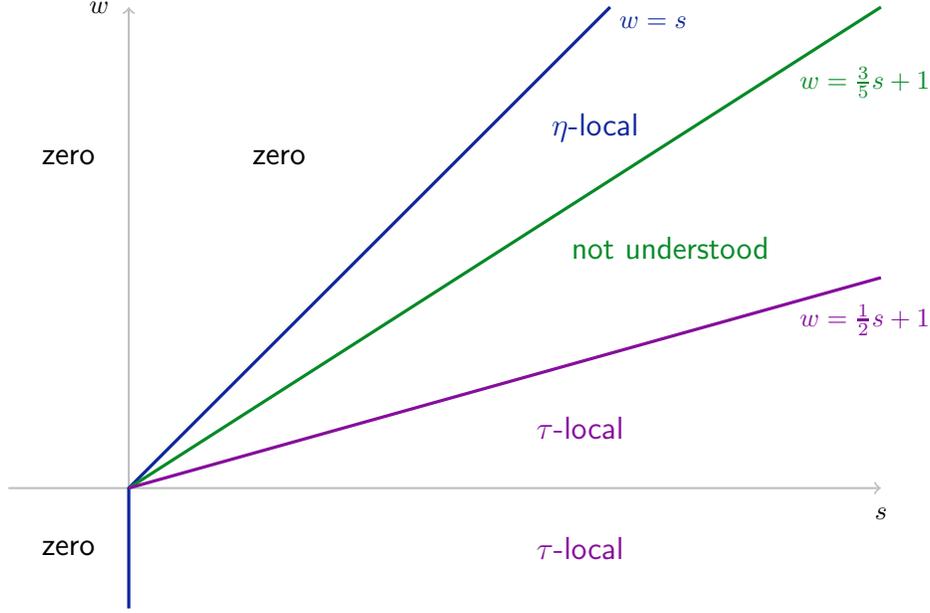
\begin{figure}[ht]
\begin{tikzpicture}[scale=0.4] 
\centering

\draw [ thick, ->, mygray] (0,-4) -- (0,16);
\node [left= 4pt] at (0,16) { $w$ };
\draw [ thick, ->, mygray] (-4,0) -- (25,0);
\node [below=4pt] at (25,0) { $s $ };

\draw [very thick, color3] (0,-4) -- (0,0) -- (16,16);
\node [right, color3] at (16,15.5) { $w = s$ };
\node [color3] at (15.5,12) { \Large $\eta\textsf{-local}$  };

\draw [very thick, color2] (0,0) -- (25,16);
\node [right, color2] at (22,13.5) { $w = \frac{3}{5}s + 1$ };
\node [color2] at (18,8) {  \Large $\textsf{not understood}$ };

\draw [very thick, color1] (0,0) -- (25,7);
\node [right, color1] at (22,5.6) { $w = \frac{1}{2}s + 1$ };
\node [color1, align=center] at (15,2) { \Large  $\tau\textsf{-local}$};
\node [color1, align=center] at (15,-2) { \Large  $\tau\textsf{-local}$};

\node at (-2,-2) { \Large \textsf{zero} };
\node at (-2,11) { \Large \textsf{zero} };
\node at (5,11) { \Large \textsf{zero} };
\end{tikzpicture}

\caption{Homotopy groups $\pi_{s,w}$ of the 2-completed
motivic sphere over $\C$ (not to scale).}
\label{fig:chartintro}
\end{figure}

Beware that the grading in Figure \ref{fig:chartintro} does not follow
the usual Adams style.  The 
vertical axis plots the motivic weight, not a spectral sequence
filtration.  On the other hand, the Adams and Adams-Novikov charts in
\cite{Isaksen-Adams} \cite{Isaksen-Adams-Novikov} and elsewhere
suppress the motivic weight and show
spectral sequence filtrations on the vertical axis.  The reader
should be careful not to confuse the different conventions.

We will now describe the qualitative behavior of 
motivic homotopy groups in each region of Figure \ref{fig:chartintro}.

\subsection*{The vanishing region}
The group $\pi_{s,w}$ is zero when $s<0$ or $w > s$.
This is closely related to Morel's connectivity theorem for 
motivic homotopy theory \cite{MorelA1}.
We will give a different proof in Theorem \ref{thm:vanishing}
that is consistent with the methods of this article.
The boundaries of the vanishing region are sharp in the sense
that there are infinitely families of non-zero elements that lie
on the line $s = 0$ and on the line $w = s$.

\subsection*{The $\tau$-local region}
Levine \cite[Theorem 6.7]{Levine14} showed that the realization functor from motivic homotopy theory over $\C$ to classical homotopy theory induces an isomorphism 
\begin{equation*}
\pi_{s,0} \stackrel{\iso}{\to} \pi_s
\end{equation*}
from motivic homotopy groups in weight zero to classical homotopy groups
(even before 2-completion).
This means that 
the classical groups $\pi_{s}$ lie on the positive $s$-axis
in Figure \ref{fig:chartintro}.

Recall that $\tau$ is an element of $\pi_{0,-1}$ that induces 
multiplication by $\tau$ in motivic cohomology.
The realization functor from motivic homotopy theory
over $\C$ to classical homotopy theory is well-understood calculationally.
According to \cite[Proposition 3.0.2]{Isaksen14}, applying this functor has the
effect of inverting $\tau$.
We use the motivic Adams-Novikov spectral sequence to show that 
the maps
\[
\pi_{s,w} \map \pi_{s,w-1}
\]
given by multiplication by $\tau$ are isomorphisms when
$w \leq \frac{1}{2} s + 1$ (and when $w \leq 0$ if $s = 0$).
Therefore, each group $\pi_{s,w}$ in this region is isomorphic 
to its $\tau$-localization, which in turn is isomorphic to the
classical (2-completed) homotopy group $\pi_{s}$.

The slope of the boundary of the $\tau$-local region 
is sharp in the sense that there is 
an infinite family of non-zero elements that lie 
on a line of slope $\frac{1}{2}$ and are annihilated by $\tau$.

\subsection*{The $\eta$-local region}
The motivic Hopf map $\eta$ in $\pi_{1,1}$ is not nilpotent.
The values of the localized groups $\pi_{\ast,\ast} [\eta^{-1}]$ 
were conjectured in \cite{GI} and proved in \cite{AM14}.
See Proposition \ref{prop:loc-pi} 
for an explicit description of the calculation.
We use results from \cite{AM14}
to show that every element of $\pi_{s,w}$ is $\eta$-local
when $w > \frac{3}{5} s + 1$ and $w \leq s$.
This means that the maps 
\[
\pi_{s,w} \map \pi_{s+1,w+1}
\]
given by multiplication by $\eta$ are isomorphisms in this region.
As a result, every group in the $\eta$-local region is explicitly known.

The upper boundary of the $\eta$-local region is sharp, in the sense that 
there is an $\eta$-local family of non-zero elements that lie 
on the upper boundary line.
The slope of the lower boundary is sharp in the sense that there is an
infinite family of non-zero elements that lie on a line
of slope $\frac{3}{5}$ and are annihilated by $\eta$.
 
\subsection*{The not understood region}
The final region consists of groups $\pi_{s,w}$ 
for which $w \leq \frac{3}{5} s + 1$ and $w > \frac{1}{2} s + 1$.
All of the exotic motivic phenomena fall within this relatively
small slice whose boundary lines meet at an angle of less than 5 degrees.
Further study of motivic stable homotopy groups over $\C$ has the
aim of better understanding this region.


\section{The classical and motivic Adams-Novikov spectral sequences}

We will work in the same framework as \cite{Isaksen14}, that is, over $\C$
and at the prime 2.
Our results extend to any other algebraically closed field of characteristic 
zero. 

Let $\clE{r}$ be the $E_r$ page of the classical Adams-Novikov spectral sequence, and let $\motE{r}$ be the $E_r$ page of the motivic Adams-Novikov spectral sequence
for the 2-completed motivic sphere.
In the motivic context,
we use degrees of the form $(s,f,w)$,
where $s$ is the stem, $f$ is the Adams-Novikov filtration,
and $w$ is the weight.
In the classical context, we use degrees of the form $(s,f)$,
where $s$ is the stem and $f$ is the Adams-Novikov filtration.
An element in degree $(s,f,w)$ occurs at location $(s,f)$ 
in a traditional Adams-Novikov chart.

We next describe the relationship between
the motivic $E_2$ page and the classical $E_2$ page.
First define an intermediate object
\begin{equation*}
\olmotE[s,f,w]{2} = \left\{ 
  \begin{array}{l l}
    \clE[s,f]{2} & \quad \text{if $ w = \frac{s + f}{2}$ }\\
    0 & \quad \text{if $ w \neq \frac{s + f}{2}$.}
  \end{array} \right.
\end{equation*}

\begin{prop}[{\cite[Theorem 8 and Section 4]{HKO11}}, {\cite[Theorem 6.1.4]{Isaksen14}}]
\label{prop:motANss}
There is a tri-graded isomorphism
\begin{equation*}
\motE{2} \cong \olmotE{2} \otimes_{\Z} \Z[\tau]
\end{equation*}
where $\tau$ has degree $(0,0,-1)$.
\end{prop}
In other words, the motivic Adams-Novikov $E_2$ page 
is completely determined by the classical 
Adams-Novikov $E_2$ page.
We rephrase Proposition \ref{prop:motANss} in an even more explicit form.

\begin{cor}
\label{cor:motANss}
As a graded abelian group, the motivic 
Adams-Novikov $E_2$ page is given by 
\begin{equation*}
\motE[s,f,w]{2} = \left\{ 
  \begin{array}{l l}
    \clE[s,f]{2} & \quad \text{if $ w \leq \frac{s + f}{2}$ }\\
    0 & \quad \text{if $ w > \frac{s + f}{2}$.}
  \end{array} \right.
\end{equation*}
\end{cor}

Let $\aa$ be the classical element of $\clE[1,1]{2}$ 
that represents the Hopf map $\eta$ in $\pi_1$.
We abuse notation and write $\aa$ for its motivic analogue in 
$\motE[1,1,1]{2}$ that represents the motivic Hopf map $\eta$ in $\pi_{1,1}$. Define $\aa^{-1}\clE[s,f]{2}$ to be 
\begin{equation*}
\colim \left( \clE[s,f]{2} \stackrel{\cdot \aa}{\lto} \clE[s+1,f+1]{2} \stackrel{\cdot \aa}{\lto} \cdots \right).
\end{equation*}
Similarly, define
$\aa^{-1}\motE[s,f,w]{2}$ to be 
\begin{equation*}
\colim \left( \motE[s,f,w]{2} \stackrel{\cdot \aa}{\lto} \motE[s+1,f+1,w+1]{2} \stackrel{\cdot \aa}{\lto} \cdots \right).
\end{equation*}
The groups $\aa^{-1}\clE[s,f]{2}$ assemble into 
a localized Adams-Novikov spectral sequence which, by a vanishing line argument, converges to the classical homotopy groups 
$\pi_{\ast} S^0 [\eta^{-1}]$ of the $\eta$-localized classical sphere.
These localized groups are zero since $\eta^4$ is zero classically.
Similarly,
the groups $\aa^{-1} \motE[s,f,w]{2}$ assemble into a 
localized motivic Adams-Novikov spectral sequence that converges to 
$\pi_{\ast,\ast} S^{0,0}[\eta^{-1}]$, which will be described below in 
Proposition \ref{prop:loc-pi}.

\begin{prop} \label{prop:locANss}
The localization map
\begin{equation*}
\motE[s,f,w]{2} \lto \aa^{-1}\motE[s,f,w]{2}
\end{equation*}
is an isomorphism for all weights $w$ and for $s < 5f - 10$.
\end{prop}

\begin{proof}
By \cite[Proposition 5.1]{AM14}, the localization map
\begin{equation*}
\clE[s, f]{2} \lto \aa^{-1}\clE[s, f]{2}
\end{equation*}
is an isomorphism in the range $s < 5f -10$.
The motivic analogue follows from the identifications
of Corollary \ref{cor:motANss}.
\end{proof}

For reference, we give a complete description of the $\aa$-localized 
motivic Adams-Novikov spectral sequence. 
We are mostly interested in the $E_{\infty}$ page, which gives the homotopy of the $\eta$-local motivic sphere.

\begin{prop}[{\cite[Corollary 7.1 and Theorem 7.2]{AM14}}]
\label{prop:loc-pi}
\mbox{}

\begin{enumerate}
\item The $E_2$ page of the localized motivic Adams-Novikov spectral sequence is
\begin{equation*}
\aa^{-1} \motE[*,*,*]{2} \cong \quotient{\F_2[\tau, \aa^{\pm 1}, \alpha_3, \alpha_4]}{\alpha_4^2},
\end{equation*}
where $\tau$ has degree $(0,0,-1)$;
$\aa$ has degree $(1,1,1)$;
$\alpha_3$ has degree $(5,1,3)$; and
$\alpha_4$ has degree $(7,1,4)$.
\item All differentials are deduced via the Leibniz rule from $d_3(\alpha_3) = \tau \alpha_1^4$, and 
\begin{equation*}
\aa^{-1} \motE[s,f,w]{\infty} \cong \quotient{\F_2[\aa^{\pm 1}, \alpha_3^2, \alpha_4]}{\alpha_4^2}.
\end{equation*}

\item The homotopy of the $\eta$-localized motivic sphere is given by
\begin{equation*}
\pi_{\ast, \ast} S^{0,0}[\eta^{-1}] \iso \quotient{\F_2[\eta^{\pm 1}, \sigma, \mu_9]}{\sigma^2},
\end{equation*}
where $\eta$ in degree $(1,1)$ is detected by $\aa$;
$\sigma$ in degree $(7,4)$ is detected by $\alpha_4$;
and $\mu_9$ in degree $(9,5)$ is detected by 
$\aa^{-1} \alpha_3^2$.
\end{enumerate}
\end{prop}


\section{The cofiber of $\tau$}

Recall that $\tau$ is a map
$S^{0,-1} \lto S^{0,0}$ 
that induces multiplication by $\tau$ in motivic cohomology.
Let $\Ct$ be the cofiber of this map. We have a cofiber sequence
\begin{equation} \label{eq:cofseqtau}
S^{0,-1} \stackrel{\tau}{\lto} S^{0,0} \lto \Ct \lto S^{1,-1}.
\end{equation}
We will use the spectrum $\Ct$ to obtain information about the homotopy groups
of the motivic sphere. 

Using the relationship between the motivic and classical Adams-Novikov 
spectral sequences, we get the following remarkable description 
of the homotopy groups of $\Ct$.

\begin{prop}[{\cite[Proposition 6.2.5]{Isaksen14}}]
\label{prop:htpycoftau}
The homotopy of $\Ct$ is given by
\begin{equation*}
\pi_{s,w}(\Ct) \iso \motE[s,2w -s,w]{2} \iso \clE[s,2w-s]{2}.
\end{equation*}
\begin{proof}
For reference, 
we reproduce the short proof from \cite{Isaksen14}.
The defining cofiber sequence \eqref{eq:cofseqtau} induces a long exact sequence 
\begin{equation*}
\ldots \lto E_2(S^{0,-1}; BPL) \stackrel{\tau}{\lto} \motE{2} \lto E_2(\Ct; BPL) \lto \cdots
\end{equation*}
on motivic Adams-Novikov $E_2$ pages.
Multiplication by $\tau$ is injective by Proposition \ref{prop:motANss}, so $E_2(\Ct;BPL)$ is isomorphic to $\olmotE{2}$. Therefore, 
$E_2(\Ct;BPL)$ is concentrated in degrees of the form $(s,f,\frac{s+f}{2})$.
For degree reasons, there can be no differentials, so the spectral sequence collapses to $E_{\infty}(\Ct; BPL) \cong \olmotE{2}$. 
Moreover, for similar degree reasons, no hidden extensions are possible.
\end{proof}
\end{prop}

Proposition \ref{prop:htpycoftau}
says essentially that the motivic homotopy groups of $\Ct$
are isomorphic to the classical Adams-Novikov $E_2$ page 
as bigraded objects.

\begin{cor} \label{cor:htpycoftau}
The 
group $\pi_{s,w}(\Ct)$ is zero when $w \leq \frac{1}{2} s$, 
except that $\pi_{0,0}(\Ct) = \Z_2$.
\begin{proof}
By Proposition \ref{prop:htpycoftau}, we have an isomorphism $\pi_{s,w}(\Ct) \cong \clE[s,2w-s]{2}$. 
When $2w-s \leq 0$, 
the group 
$\clE[s,2w-s]{2}$ is zero, except that
$\clE[0,0]{2}$ equals $\Z_2$.
\end{proof}
\end{cor}


\section{The vanishing and local regions} \label{section:mainresults}
 
The goal of this section is to make precise the 
results on each region of Figure \ref{fig:chartintro} that were
described informally in Section \ref{sctn:intro}.

\subsection{The vanishing region}

We first describe the region to the left of the line $s =0$
or above the line $w = s$.

\begin{thm}
\label{thm:vanishing}
The group $\pi_{s,w}$ is zero if $w > s$ or $s <0$.
\begin{proof}
We begin with the motivic May spectral sequence; see 
\cite[Chapter 2]{Isaksen14} for more details.
The $E_1$ page is a polyonomial algebra over $\F_2[\tau]$ 
with generators $h_{ij}$ for all $i > 0$ and $j \geq 0$.
The element $h_{i0}$ lies in stem $2^i-2$ and weight $2^{i-1} - 1$,
while the element $h_{ij}$ lies in stem $2^j(2^i - 1) -1$
and weight $2^{j-1} (2^i-1)$ for $j > 0$.
In every case, the weight is always less than or equal to the stem.
Therefore, the entire May $E_1$ page vanishes in degrees
where the weight is greater than the stem.

The target of the motivic May spectral sequence is the $E_2$ page
of the motivic Adams spectral sequence.  Thus the Adams $E_2$ page
also vanishes in degrees where the weight is greater than the stem.

The target of the motivic Adams spectral sequence are the
homotopy groups of the 2-completed motivic sphere.  Finally,
these homotopy groups vanish in degrees where the weight is greater
than the stem.

A similar argument shows that the homotopy groups vanish in negative stems.
\end{proof}
\end{thm}

\begin{remark}
Both of the boundary lines in Theorem \ref{thm:vanishing} are sharp
in the following sense.  
The non-zero elements $\eta^k$ form an infinite family on the line $w = s$,
and the non-zero elements $\tau^k$ form an infinite family on the
line $s = 0$.
\end{remark}

\subsection{The $\tau$-local region} 
We next describe the region below the line $w = \frac{1}{2}s + 1$
and to the right of the line $s = 0$.

\begin{thm}
\label{thm:tau-local}
If $w \leq \frac{1}{2} s + 1$ and $s > 0$, 
or if $w \leq 0$ and $s = 0$, then the map
\begin{equation*}
\pi_{s,w} \lto \pi_{s,w-1}
\end{equation*}
given by multiplication by $\tau$ is an isomorphism. 
Moreover, in this range $\pi_{s,w}$ is isomorphic to the classical group $\pi_s$.
\begin{proof} 
For $s= 0$, we know that $\pi_{0, \ast}$ is equal to $\Z_2[\tau]$,
where $\tau$ has degree $(0,-1)$. In Figure \ref{fig:chartintro},
these elements are represented by countably many copies of $\Z_2$
on the negative $w$ axis. Since the classical group $\pi_0$ is 
isomorphic to $\Z_2$, this settles the case $s=0$.

Now we may assume that $s > 0$. Consider the long exact sequence 
\begin{equation*}
\cdots \lto \pi_{s+1,w-1}(\Ct) \lto \pi_{s,w} \stackrel{\tau}{\lto} 
\pi_{s, w-1} \lto \pi_{s,w-1}(\Ct) \lto \cdots
\end{equation*}
of homotopy groups obtained from the cofiber sequence \eqref{eq:cofseqtau}.
By the vanishing result of Corollary \ref{cor:htpycoftau}, 
the long exact sequence becomes
\begin{equation}
\cdots \lto 0 \lto \pi_{s,w} \stackrel{\tau}{\lto} \pi_{s, w-1} 
\lto 0 \lto \cdots
\end{equation}
when $w \leq \frac{1}{2} s + 1$.
Therefore multiplication by $\tau$ is an isomorphism in this range. 

The final claim follows from the fact that 
realization from motivic homotopy theory to 
classical homotopy theory induces $\tau$-localization 
on homotopy groups \cite[Proposition 3.0.2]{Isaksen14}.
\end{proof}
\end{thm}

\begin{example}
\label{ex:Pkh14}
Consider the elements 
$P^k h_1^4$ of the motivic Adams spectral sequence
in degrees $(4,4,4) + k(8,4,4)$.
These elements detect
homotopy classes that are annihilated by $\tau$.
This family lies on the line
$w = \frac{1}{2} s + 2$.
Therefore, the slope of the line in 
Theorem \ref{thm:tau-local} cannot be improved.
\end{example}

\subsection{The $\eta$-local region} 
Now we consider the region above the line $w = \frac{3}{5} s + 1$
and below the line $w = s$.

\begin{thm}
\label{thm:eta-local}
If $w > \frac{3}{5}s + 1$, then the map
\begin{equation*}
\pi_{s,w} \lto \pi_{s+1,w+1}
\end{equation*}
given by multiplication by $\eta$ is an isomorphism.
\begin{proof}
By Proposition \ref{prop:motANss}, any non-trivial element in the motivic Adams-Novikov spectral sequence
satisfies the inequality $s + f - 2w \geq 0$. 
Under this condition, our constraint $w > \frac{3}{5}s + 1$ implies 
that $s < 5f - 10$. By Proposition \ref{prop:locANss}, in this region the motivic Adams-Novikov spectral sequence agrees with the $\aa$-localized motivic Adams-Novikov spectral sequence, and thus the survivors are $\eta$-local as claimed. 
\end{proof}
\end{thm}

\begin{remark}
In the motivic Adams spectral sequence,
every element above 
the line $f = \frac{1}{2}s + 2$ is $h_1$-local \cite{GI2}.
Using that the motivic Adams spectral sequence vanishes when
$s+f-2w < 0$ \cite[Remark 2.1.13]{Isaksen14}, we can conclude that
$\pi_{s,w}$ is $\eta$-local if $w \geq \frac{3}{4} s + 1$.
This result is strictly weaker than the result of
Theorem \ref{thm:eta-local}.
\end{remark}

\begin{remark} \label{rem:etalocalelem}
Proposition \ref{prop:loc-pi} explicitly describes
the homotopy groups of the $\eta$-local sphere. 
In particular, every homotopy group in
the $\eta$-local region is completely understood.
\end{remark}

\begin{example} \label{ex:w1periodicstuff}
Recent work of Michael Andrews and others shows that the elements 
$h_2^3 g^k$ of the motivic Adams spectral sequence in degrees
$(9,3,6) + k(20,4,12)$ are non-trivial permanent cycles
for all $k \geq 0$.
The associated homotopy classes form a ``$w_1$-periodic" 
family that lie on the line $w = \frac{3}{5} s + \frac{3}{5}$, and they 
are all annihilated by $\eta$.
This shows that the slope of the bottom line in 
Theorem \ref{thm:eta-local} cannot be improved.

On the other hand, the elements $\eta^k$ lie on the line $w = s$,
which shows that the top line in Theorem \ref{thm:eta-local}
is sharp.
\end{example}

\begin{remark}
Theorem \ref{thm:eta-local} shows that every element that lies 
above the line $w = \frac{3}{5}s + 1$ and below the line $w = s$
is $\eta$-local.  However, there are $\eta$-local elements that lie
outside this region.  For example, consider the elements $P^k h_1$ 
of the motivic Adams spectral sequence in degree
$(1,1,1) + k(8,4,4)$.
These elements detect homotopy classes that are $\eta$-local
and lie below the line $w = \frac{3}{5} s + 1$.
\end{remark}


\subsection{The not understood region} 
\label{section:exoticregion}

Finally, we come to the region above the line $w= \frac{1}{2} s$
and below the line $w = \frac{3}{5} s + 1$.
All of the exotic behavior of motivic homotopy groups lies
in this region.

For example, there is an exotic non-nilpotent element in
$\pi_{32,18}$.  This element and all of its powers lie in the
not understood region.

Motivic 
$v_n$-self maps of period $k$ have
bidegree $k( 2^{n+1}-2, 2^n-1 )$. Therefore,
all $v_n$-periodic families in $\pi_{*,*}$ lie on 
lines of slope $1/2$.  These families are parallel
to the bottom edge of the not understood region.

Recent work of Andrews and others explores ``$w_n$-periodicity"
in motivic homotopy groups.  
Motivic $w_n$-self maps of period $k$ have
bidegree $k (2^{n+2} -3, 2^{n+1} -1)$.
For example, $w_0$ has degree $(1,1)$ and in fact is the same as
multiplication by $\eta$.
Next, $w_1$ has degree $(5,3)$.  Thus, $w_1$-periodic families,
such as the one in Example \ref{ex:w1periodicstuff},
lie on lines of slope $\frac{3}{5}$.
Since $w_2$ has degree $(13,7)$, we speculate that
there is a line of slope $\frac{7}{13}$ such that
all elements above this line are $w_0$-periodic (i.e., $\eta$-local) or 
$w_1$-periodic.

As $n$ increases, $w_n$-periodic families lie on lines whose slopes
approach $\frac{1}{2}$ from above.  It is conceivable that
$w_n$-periodicity will lead to a better understanding of the
large-scale structure of the not understood region.

\bibliographystyle{amsplain}
\bibliography{mybibliography}

\providecommand{\bysame}{\leavevmode\hbox to3em{\hrulefill}\thinspace}
\providecommand{\MR}{\relax\ifhmode\unskip\space\fi MR }
\providecommand{\MRhref}[2]{%
  \href{http://www.ams.org/mathscinet-getitem?mr=#1}{#2}
}
\providecommand{\href}[2]{#2}
\begin{thebibliography}{10}

\bibitem{AM14}
Michael Andrews and Haynes Miller, \emph{Inverting the {H}opf map in the
  {A}dams-{N}ovikov spectral sequence}, preprint (2014).

\bibitem{GI2}
Bertrand Guillou and Daniel~C. Isaksen, \emph{The motivic vanishing line of
  slope $1/2$}, arXiv:1501.02872 (2015).

\bibitem{GI}
\bysame, \emph{The $\eta$-local motivic sphere}, J. Pure Appl. Algebra (to
  appear).

\bibitem{HKO11b}
P.~Hu, I.~Kriz, and K.~Ormsby, \emph{Convergence of the motivic {A}dams
  spectral sequence}, J. K-Theory \textbf{7} (2011), no.~3, 573--596.
  \MR{2811716 (2012h:14054)}

\bibitem{HKO11}
Po~Hu, Igor Kriz, and Kyle Ormsby, \emph{Remarks on motivic homotopy theory
  over algebraically closed fields}, J. K-Theory \textbf{7} (2011), no.~1,
  55--89. \MR{2774158 (2012b:14040)}

\bibitem{Isaksen-Adams}
Daniel~C. Isaksen, \emph{Classical and motivic adams charts}, arXiv:1401.4983
  (2014).

\bibitem{Isaksen-Adams-Novikov}
\bysame, \emph{Classical and motivic adams-novikov charts}, arXiv:1408.0248
  (2014).

\bibitem{Isaksen14}
\bysame, \emph{Stable stems}, arXiv:1407.8418 (2014).

\bibitem{Levine14}
Marc Levine, \emph{A comparison of motivic and classical stable homotopy
  theories}, J. Topol. \textbf{7} (2014), no.~2, 327--362. \MR{3217623}

\bibitem{MorelA1}
Fabien Morel, \emph{The stable {$\mathbb{A}^1$}-connectivity theorems},
  $K$-Theory \textbf{35} (2005), no.~1-2, 1--68. \MR{2240215 (2007d:14041)}

\end{thebibliography}

\end{document}